\documentclass[12pt,a4paper]{article}%
\usepackage{amsmath, amsfonts, amsthm,color, latexsym, amssymb}
\usepackage{amscd}
\usepackage{amsmath}
\usepackage{amsfonts}
\usepackage{amssymb}
\usepackage{graphicx}%
\usepackage[active]{srcltx}
\providecommand{\U}[1]{\protect\rule{.1in}{.1in}}
\providecommand{\U}[1]{\protect\rule{.1in}{.1in}}
\providecommand{\U}[1]{\protect\rule{.1in}{.1in}}
\newtheorem{theorem}{Theorem}[section]
\newtheorem{proposition}[theorem]{Proposition}
\newtheorem{corollary}[theorem]{Corollary}
\newtheorem{example}[theorem]{Example}

\newtheorem{remark}[theorem]{Remark}

\newtheorem{lemma}[theorem]{Lemma}
\newtheorem{definition}[theorem]{Definition}
\newcommand{\Hhat}[1]{\stackrel{\begin{picture}(60,5)\put(0,0){\line(6,1){30}}\put(30,5){\line(6,-1){30}}\end{picture}}{#1}}
\newcommand{\hd}[1]{{\hat{d}}^}
\begin{document}

\title{Hypercyclicity of convolution operators on spaces of entire functions }
\author{F. J. Bertoloto\thinspace, G. Botelho\thanks{Supported by CNPq Grant
306981/2008-4 and INCT-Matem\'{a}tica.}\,\,, V. V.
F\'{a}varo\thanks{Supported by FAPEMIG Grant CEX-APQ-00208-09.\newline 2010 Mathematics Subject Classification: 32DXX, 47A16, 46G20.}\,\,, A.
M. Jatob\'{a}}
\date{}
\maketitle

\begin{abstract}
In this paper we use Nachbin's holomorphy types to generalize some recent results concerning hypercyclic convolution operators on Fr\'echet spaces of entire functions  of bounded type of infinitely many complex variables.

\end{abstract}

\section{Introduction}\label{sec1}
A mapping $f \colon X \longrightarrow X$, where $X$ is a
topological space, is {\it hypercyclic} if the set \linebreak
$\{x, f(x), f^2(x), \ldots\}$ is dense in $X$ for some $x \in X$.
In this case, $x$ is said to be a \emph{hypercyclic
vector for $f$}.

The study of hypercyclic translation and differentiation operators on spaces of entire
 functions of one complex variable can  be traced back to Birkhoff \cite{birkhoff}
 and MacLane \cite{maclane}. Godefroy and Shapiro \cite{godefroy} pushed these
  results quite further by proving that every convolution operator on spaces
  of entire functions of several complex variables which is not a scalar multiple of
  the identity is hypercyclic. Results on the hypercyclicity of convolution operators
  on spaces of entire functions of infinitely many complex variables appeared
  later (see, e.g., \cite{aron, andre, peterssonpascal,peterssonjmaa}).
   Recently, Carando, Dimant and Muro \cite{CDSjmaa} proved some far-reaching
   results - including a solution to a problem posed in \cite{aronmarkose} -
    that encompass as particular cases several of the above mentioned results. The main tool they use are the so-called {\it coherent sequences of homogeneous polynomials}, introduced by themselves in \cite{nachri} based on properties of polynomials ideals previously studied in \cite{indagationes, studia}.

The aim of this paper is to generalize the results of \cite{CDSjmaa}. We accomplish this task by proving results (Theorems \ref{main1} and \ref{main2}) of which the main results of {\cite{CDSjmaa} (\cite[Theorem 4.3]{CDSjmaa} and \cite[Corollary 4.4]{CDSjmaa}) are particular cases. Furthermore we give some concrete examples (Example \ref{exemplos}) that are covered by our results but not by the results of \cite{CDSjmaa}. Being strictly more general than the results of \cite{CDSjmaa}, our results also generalize the ones first generalized by \cite{CDSjmaa}.

Our approach differs from the approach of \cite{CDSjmaa} in our use of holomorphy types (in the sense of Nachbin \cite{Nachbin}) instead of coherent sequences of polynomials. More precisely, we use the $\pi_1$-$\pi_2$-holomorphy types introduced by the third and fourth authors in \cite{favaro-jatoba1}. Although we already knew that $\pi_1$-$\pi_2$-holomorphy types could be used in this context, it was only reading \cite{CDSjmaa} that we realized that the original definitions could be refined (see Definition \ref{pi-tipo de holomorfia}) to prove such general results on the hypercyclicity of convolution operators on spaces of entire functions.  Holomorphy types are a somewhat old-fashioned topic in infinite-dimensional analysis, so it is quite surprising that our holomorphy type-oriented-approach turned out to be more effective than the coherent sequence-oriented-approach.

The paper is organized as follows: in Section 2 we state our main results, in Section 3 we prove that our results are more general - not only formally but also concretely - than the results of \cite{CDSjmaa}, and in Section 4 we prove our main results. In Section 5 we extend to our context some related results that appeared in the literature, including results on surjective hypercyclic convolution operators and connections with the existence of dense subspaces formed by hypercyclic functions for convolution operators.

Throughout the paper $\mathbb{N}$ denotes the set of positive integers and
$\mathbb{N}_{0}$ denotes the set $\mathbb{N}\cup\{0\}$. The letters $E$ and
$F$ will always denote complex Banach spaces and $E^{\prime}$ represents the
topological dual of $E$. The Banach space of all continuous $m$-homogeneous
polynomials from $E$ into $F$ endowed with its usual sup norm is denoted by $\mathcal{P}(^{m}E;F)$. The subspace of $\mathcal{P}(^{m}E;F)$ of all polynomials of finite type is represented by $\mathcal{P}_f(^{m}E;F)$. The linear
space of all entire mappings from $E$ into $F$ is denoted by $\mathcal{H}%
(E;F)$. When $F=\mathbb{C}$ we write $\mathcal{P}(^{m}E)$, $\mathcal{P}_f(^{m}E)$  and $\mathcal{H}(E)$
instead of $\mathcal{P}(^{m}E;\mathbb{C})$, $\mathcal{P}_f(^{m}E;\mathbb{C})$  and $\mathcal{H}(E;\mathbb{C})$, respectively. For the general theory of homogeneous polynomials and holomorphic functions we refer to Dineen \cite{dineen} and Mujica \cite{mujica}.

\section{Main results}
In this section we state the main results of the paper and give the definitions needed to understand them.

\begin{definition}\rm (Nachbin \cite{Nachbin})
\label{Definition tipo holomorfia} A \emph{holomorphy type} $\Theta$
from $E$ to $F$ is a sequence of Banach spaces $(\mathcal{P}_{\Theta}%
(^{m}E;F))_{m=0}^{\infty}$, the norm on each of them being denoted by
$\|\cdot\|_{\Theta}$, such that the following conditions hold true:

\item[$(1)$] Each $\mathcal{P}_{\Theta}(^{m}E;F)$ is a linear
subspace of $\mathcal{P}(^{m}E;F)$.

\item[$(2)$] $\mathcal{P}_{\Theta}(^{0}E;F)$ coincides with
$\mathcal{P}(^{0}E;F)=F$ as a normed vector space.

\item[$(3)$] There is a real number $\sigma\geq1$ for which the
following is true: given any $k\in\mathbb{N}_{0}$, $m\in\mathbb{N}_{0}$,
$k\leq m $, $a\in E$ and $P\in\mathcal{P}_{\Theta}(^{m}E;F)$, we have
\[
\hat{d}^{k}P(a)\in\mathcal{P}_{\Theta}(^{k}E;F) {\rm ~~and}
\]
\[
\left\|  \frac{1}{k!}\hat{d}^{k}P(a)\right\|  _{\Theta}\leq\sigma
^{m}\|P\|_{\Theta}\|a\|^{m-k}.
\]
\end{definition}
\noindent It is plain that each inclusion $\mathcal{P}_{\Theta}(^{m}E;F)\subseteq
\mathcal{P}(^{m}E;F)$ is continuous and that $\|P\|\leq\sigma^{m}\|P\|_{\Theta}$ for every $P\in\mathcal{P}_{\Theta}(^{m}E;F)$.

\begin{definition}\rm (Gupta \cite{apostilagupta, gupta})
\label{Definition f holomorfia} Let $(\mathcal{P}_{\Theta}%
(^{m}E;F))_{m=0}^{\infty}$ be a holomorphy type from $E$ to $F$. A given
$f\in\mathcal{H}(E;F)$ is said to be of \emph{$\Theta$-holomorphy type of
bounded type} if

\item[$(1)$] $\frac{1}{m!}\hat{d}^{m}f(0)\in\mathcal{P}_{\Theta}%
(^{m}E;F)$ for all $m\in\mathbb{N}_{0}$,

\item[$(2)$] $\displaystyle\lim_{m \rightarrow\infty} \left(  \frac{1}{m!}%
\|\hat{d}^{m}f(0)\|_{\Theta}\right)  ^{\frac{1}{m}}=0.$

\noindent The linear subspace of $\mathcal{H}(E;F)$ of all functions $f$ of $\Theta
$-holomorphy type of bounded type is denoted by $\mathcal{H}_{\Theta b}(E;F)$.

\end{definition}

\begin{remark}\rm (a) The inequality $\|\cdot\|\leq\sigma^{m}\|\cdot\|_{\Theta}$ implies
that each entire mapping $f$ of $\Theta$-holomorphy type of bounded type is an
entire mapping of bounded type in the sense of Gupta in
\cite{gupta}, that is, $f$ is bounded on bounded subsets of $E$. \\
(b) It is clear that $\mathcal{P}_{\Theta}(^{m}E;F) \subseteq
\mathcal{H}_{\Theta b}(E;F)$ for each $m\in\mathbb{N}_{0}$.
\end{remark}


For each $\rho>0,$ condition (2) of Definition \ref{Definition f holomorfia} guarantees that the correspondence
\[
f\in\mathcal{H}_{\Theta b}(E;F)\mapsto \Vert f\Vert_{\Theta,\rho}%
=\sum_{m=0}^{\infty}\frac{\rho^{m}}{m!}\Vert\hat{d}^{m}f(0)\Vert_{\Theta
}<\infty\]
is a well defined seminorm on $\mathcal{H}_{\Theta
b}(E;F)$.
We shall henceforth consider $\mathcal{H}_{\Theta b}(E;F)$ endowed with the locally convex topology generated by the
seminorms $\Vert\cdot\Vert_{\Theta,\rho},$ $\rho>0.$ This topology shall be denoted
by $\tau_{\Theta}$. It is well known that
$(\mathcal{H}_{\Theta b}(E;F),\tau_{\Theta})$ is a
Fr\'{e}chet space (see, e.g, \cite[Proposition 2.3]{favaro-jatoba1}).

Next definitions are refinements of the concepts of $\pi_1$-holomorphy type and $\pi_2$-holomorphy type introduced in \cite{favaro-jatoba1}.


\begin{definition}\rm (a)
\label{pi-tipo de holomorfia} A holomorphy type $(\mathcal{P}_{\Theta}%
(^{m}E;F))_{m=0}^{\infty}$ from $E$ to $F$ is said to be a \emph{$\pi_{1}$-holomorphy type} if the following conditions hold:

\item[(a1)] Polynomials of finite type belong to $(\mathcal{P}_{\Theta}%
(^{m}E;F))_{m=0}^{\infty}$ and there exists $K>0$ such that $$\Vert\phi^{m}\cdot b\Vert_{\Theta
}\leq K^{m}\Vert\phi\Vert^{m}\cdot\Vert b\Vert$$ for all $\phi\in E^{\prime}$,
$b\in F$ and $m\in\mathbb{N}$;

\item[(a2)] For each $m\in\mathbb{N}_{0}$, $\mathcal{P}_{f}(^{m}E;F)$ is dense in $(\mathcal{P}_{\Theta}(^{m}E;F),\Vert\cdot\Vert_{\Theta})$.

\medskip

\noindent (b) A holomorphy type $(\mathcal{P}_{\Theta}(^{m}E))_{m=0}^{\infty}$ from $E$ to $\mathbb{C}$ is said to be
a \emph{$\pi_{2}$-holomorphy type} if for each $T\in\left[
\mathcal{H}_{\Theta b}(E)\right]  ^{\prime}$, $m\in\mathbb{N}_{0}$ and
$k\in\mathbb{N}_{0},$ $k\leq m$, the following conditions hold:
\item[(b1)] If $P\in
\mathcal{P}_{\Theta}(^{m}E)$ and $A \colon E^m \longrightarrow \mathbb{C}$ is the unique continuous symmetric $m$-linear mapping such that $P=\hat{A},$ then the
$\left(  m-k\right)  $-homogeneous polynomial%
\begin{align*}
T\left(  \widehat{A(\cdot)^{k}}\right)  \colon E &  \longrightarrow
\mathbb{C}\\
y &  \mapsto T\left(  A(\cdot)^{k}y^{m-k}\right)
\end{align*}
belongs to $\mathcal{P}_{\Theta}(^{m-k}E);$
\item[(b2)] For constants $C, \rho > 0$ such that
\[
\left\vert T\left(  f\right)  \right\vert \leq C\left\Vert f\right\Vert
_{\Theta,\rho}
{\rm ~for~every~}f\in\mathcal{H}_{\Theta b}(E),\]
which exist because $T\in\left[
\mathcal{H}_{\Theta b}(E)\right]  ^{\prime}$, there is a constant $K>0$ such that \[
\Vert T(\widehat{A(\cdot)^{k}})\Vert_{\Theta}\leq C\cdot K^{m}\rho^{k}\Vert
P\Vert_{\Theta} {\rm ~for~ every~}
P\in\mathcal{P}_{\Theta}(^{m}E).\]

\end{definition}

\begin{definition}\rm Let $\Theta$ be a holomorphy type from $E$ to $\mathbb{C}$. \\
(a) For $a\in E$ and
$f\in\mathcal{H}_{\Theta b}(E)$, the {\it translation of $f$ by $a$} is defined by
$$\tau_a f \colon E \longrightarrow \mathbb{C}~,~\left(  \tau_{a}f\right)  \left(  x\right)  =f\left(
x-a\right)  .$$ By \cite[Proposition 2.2]{favaro-jatoba1} we have $\tau_{a}f\in\mathcal{H}_{\Theta
b}(E).$\\
(b) A continuous linear operator $L\colon\mathcal{H}_{\Theta
b}(E)\longrightarrow\mathcal{H}_{\Theta b}(E)$ is called a \emph{convolution
operator on }$\mathcal{H}_{\Theta b}(E)$ if it is translation invariant, that
is,
$$L(\tau_{a}f)=\tau_{a}(L(f))$$
for all $a\in E$ and $f\in\mathcal{H}_{\Theta b}(E).$\\
(c) For each functional $T \in [\mathcal{H}_{\Theta b}(E)]^{\prime}$, the operator $\bar{\Gamma}_{\Theta}(T)$ is defined by
$$\bar{\Gamma}_{\Theta}(T)\colon
\mathcal{H}_{\Theta b}(E)\longrightarrow \mathcal{H}_{\Theta
b}(E)~,~\bar{\Gamma}_{\Theta}(T)(f)=T\ast
f, $$
where the {\it convolution product} $T\ast f$ is defined by
\[
\left(  T\ast f\right)  \left(  x\right)  =T\left(  \tau_{-x}f\right)
{\rm ~for~ every~} x\in E.\]
(d) $\delta_0\in[\mathcal{H}_{\Theta b}(E)]^\prime$ is the linear functional defined by
$$\delta_0\colon\mathcal{H}_{\Theta b}(E)\longrightarrow\mathbb{C}~,~\delta_0(f)=f(0).$$
\end{definition}

The main results of this paper read as follows:

\begin{theorem}\label{main1} Let $E^{\prime}$ be
separable and $(\mathcal{P}_{\Theta}(^{m}E))_{m=0}^{\infty}$ be a
$\pi_1$-holomorphy type from $E$ to $\mathbb{C}$. Then every convolution operator
on $\mathcal{H}_{\Theta b}(E)$ which is not a scalar multiple of the
identity is hypercyclic.
\end{theorem}

\begin{theorem}\label{main2}
Let $E^{\prime}$ be
separable, $(\mathcal{P}_{\Theta}(^{m}E))_{m=0}^{\infty}$ be a
$\pi_1$-$\pi_2$-holomorphy type and $T\in[\mathcal{H}_{\Theta
b}(E)]^{\prime}$ be a linear functional which is not a scalar multiple of $\delta_0$. Then $\bar{\Gamma}_{\Theta}(T)$ is a convolution operator that is not a scalar multiple of the identity, hence hypercyclic.
\end{theorem}

\section{Comparison with known results}
Before proceeding to the proofs we shall establish that Theorems \ref{main1} and \ref{main2} are strictly more general than \cite[Theorem 4.3]{CDSjmaa} and \cite[Corollary 4.4]{CDSjmaa}, respectively. First of all we have to give the definitions needed to understand these results from \cite{CDSjmaa}.

\begin{definition}\rm For $P\in \mathcal{P}(^kE)$, $a\in E$ and $r\in\mathbb{N}$, $P_{a^r}$ denotes the
$(k-r)$-homogeneous polynomial on $E$ defined by

$$P_{a^r}(x)=A(\underbrace{a,\ldots,a}_{r\, \rm{times}},x,\ldots,x),$$
where, as before, $A$ stands for the continuous symmetric $k$-linear form such that $P(x) = A(x,\ldots, x)$ for every $x \in E$.
\end{definition}

\begin{definition}\label{coherent sequence}\rm (Carando, Dimant, Muro \cite{nachri, CDSjmaa}) For each $k \in \mathbb{N}_0$, $\mathfrak{A}_k(E)$ and $\mathfrak{B}_k(E)$ are linear subspaces of ${\cal P}(^k E)$ containing the polynomials of finite type which are Banach spaces with norms $\| \cdot \|_{\mathfrak{A}_k(E)}$ and $\| \cdot \|_{\mathfrak{B}_k(E)}$, respectively. $\mathfrak{A}_k(E)$ and ${\mathfrak{B}_k(E)}$ are also asked to be continuously contained in ${\cal P}(^k E)$.

\label{coherent} A sequence $\mathfrak{A}\left( E\right) =\left\{
\mathfrak{A}_{k}\left(  E\right)  \right\} _{k \in \mathbb{N}_0}$ is said to be {\rm a coherent sequence of homogeneous polynomials} if there exist positive constants $C$ and $D$ such
that the following conditions hold for all $k$:

\item[(a)] For each $P\in\mathfrak{A}_{k+1}\left(  E\right)  $ and $a\in E$,
$P_{a}\in\mathfrak{A}_{k}\left(  E\right)  $ and
\[
\Vert P_{a}\Vert_{\mathfrak{A}_{k}\left(  E\right)  }\leq C\Vert
P\Vert_{\mathfrak{A}_{k+1}\left(  E\right)  }\Vert a\Vert.
\]

\item[(b)] For each $P\in\mathfrak{A}_{k}\left(  E\right)  $ and $\gamma\in
E^{\prime}$, $\gamma P\in\mathfrak{A}_{k+1}\left(  E\right)  $ and
\[
\Vert\gamma P\Vert_{\mathfrak{A}_{k+1}\left(  E\right)  }\leq
D\Vert \gamma\Vert\Vert P\Vert_{\mathfrak{A}_{k}\left(  E\right)
}.
\]
As usual, for $k=0$, $\mathfrak{A}_{0}\left(  E\right) $ is the
$1$-dimensional space of constant functions on $E$, that is $\mathfrak{A}_{0}\left(  E\right) = \mathbb{C}.$

The coherent sequence $\mathfrak{A}\left(  E\right)  =\left\{  \mathfrak{A}%
_{k}\left(  E\right)  \right\}  _{k}$ is said to be {\it multiplicative} if
there exists $M>0$ such that $PQ\in\mathfrak{A}%
_{k+l}\left(  E\right)  $ and%
\[
\Vert PQ\Vert_{\mathfrak{A}_{k+l}\left(  E\right)  }\leq
M^{k+l}\Vert
P\Vert_{\mathfrak{A}_{k}\left(  E\right)  }\Vert Q\Vert_{\mathfrak{A}%
_{l}\left(  E\right)  },\]
whenever $P\in\mathfrak{A}_{k}\left(
E\right)  $ and
$Q\in\mathfrak{A}_{l}\left(  E\right) $.
\end{definition}

\begin{remark}\label{rema}\rm Note that the case $k=0$ implies that the constant $C$ of condition \ref{coherent}(a) is
greater than or equal to $1$. From \cite[Theorem 3.2]{studia} it
follows that every coherent sequence $\left\{
\mathfrak{A}_{k}\left(  E\right)  \right\} _{k \in
\mathbb{N}_0}$is a holomorphy type with constant $\sigma = C$. So,
$$\Vert P\Vert\leq C^{k}\Vert P\Vert_{\mathfrak{A}%
_{k}\left(  E\right)  }$$ for all $P\in\mathfrak{A}_{k}\left(
E\right)  $ and $k \in \mathbb{N}_0$.
\end{remark}

Let $\mathfrak{A}\left(  E\right)  =\left\{
\mathfrak{A}_{k}\left(  E\right) \right\}  _{k}$ be a coherent sequence of homogeneous polynomials on $E$. Since $\mathfrak{A}\left(  E\right)$ is a holomorphy type by Remark \ref{rema}, we can consider the space  $\mathcal{H}_{\mathfrak{A}\left(  E\right) b}(E)$ of holomorphic functions of $\mathfrak{A}\left(  E\right)$-holomorphy type of
bounded type according to Definition
\ref{Definition f holomorfia}. Following the notation of \cite{CDSjmaa} we shall henceforth represent this space by the symbol $\mathcal{H}%
_{b\mathfrak{A}}(E)$. So $\mathcal{H}_{b\mathfrak{A}}(E)$ becomes a Fr\'{e}chet space with
the topology generated by the family of seminorms $\left\{  p_{\rho}\right\}
_{\rho>0}$, where%
\[
p_{\rho}\left(  f\right)
=\sum_{k=0}^{\infty}\frac{\rho^{k}}{k!}\Vert\hat
{d}^{k}f(0)\Vert_{\mathfrak{A}_{k}\left(  E\right)  },
\]
for $f\in\mathcal{H}_{b\mathfrak{A}}(E).$

\bigskip

Next we define the polynomial Borel transform in the context of coherent sequences:

\begin{definition}\rm Let $\mathfrak{A}\left(  E\right)  =\left\{
\mathfrak{A}_{k}\left(  E\right) \right\}  _{k}$ be a coherent
sequence. For each $k $ the \emph{polynomial Borel
transform} is defined by
$$B_{k}\colon\mathfrak{A}_{k}\left( E\right)
^{\prime}\longrightarrow\mathcal{P}\left(  ^{k}E^{\prime}\right) ~,~B_{k}\left(  \varphi\right)  \left(  \gamma\right)
=\varphi\left( \gamma^{k}\right)  .$$
From now on, the
expression $\mathfrak{A}_{k}\left(  E\right)  ^{\prime}=\mathfrak{B}%
_{k}\left(  E^{\prime}\right)  $ will always mean that the
polynomial Borel transform $B_{k}\colon\mathfrak{A}_{k}\left(
E\right)  ^{\prime }\longrightarrow\mathfrak{B}_{k}\left(
E^{\prime}\right)  $ is an isometric isomorphism.
\end{definition}


The main hypercyclicity results of \cite{CDSjmaa} are the following:

\begin{theorem}\label{cdm1}{\rm \cite[Theorem 4.3]{CDSjmaa}}
Suppose that $E^{\prime}$ is separable. Let
$\{\mathfrak{B}_{k}(E^{\prime})\}_{k}$ be a coherent sequence and
$\{\mathfrak{A}_{k}(E)\}_{k}$ be such that
$\mathfrak{A}_{k}(E)^{\prime }=\mathfrak{B}_{k}(E^{\prime})$ for
every $k$. Then, every convolution operator on
$\mathcal{H}_{b\mathfrak{A}}(E)$ which is not a scalar multiple
of the identity is hypercyclic.
\end{theorem}

\begin{corollary}\label{cdm2}{\rm \cite[Corollary 4.4]{CDSjmaa}}
Suppose that $E^{\prime}$ is separable. Let
$\{\mathfrak{B}_{k}(E^{\prime})\}_{k}$ be a coherent
multiplicative sequence and $\{\mathfrak{A}_{k}(E)\}_{k}$ be such
that $\mathfrak{A}_{k}(E)^{\prime }=\mathfrak{B}_{k}(E^{\prime})$
for every $k$. For every
$\varphi\in [\mathcal{H}_{
b\mathfrak{A}}(E)]^{\prime}$ which is not a scalar multiple of
$\delta_0$, the operator
\begin{eqnarray*}
T_\varphi\colon
\mathcal{H}_{b\mathfrak{A}}(E)\longrightarrow \mathcal{H}_{b\mathfrak{A}}(E)~,~
T_\varphi(f)=\varphi\ast f,
\end{eqnarray*}
is hypercyclic.
\end{corollary}

The next result proves that Theorem \ref{cdm1} is a particular case of Theorem \ref{main1}:

\begin{proposition}
\label{coerente_pi1}Let $\{\mathfrak{B}_{k}(E^{\prime})\}_{k}$ be a coherent sequence
and $\{\mathfrak{A}_{k}(E)\}_{k}$ be such that
$\mathfrak{A}_{k}(E)^{\prime
}=\mathfrak{B}_{k}(E^{\prime})$ for all $k$. Then $\{\mathfrak{A}%
_{k}(E)\}_{k}$ is a $\pi_{1}$-holomorphy type from $E$ to $\mathbb{C}$.
\end{proposition}

\begin{proof}
By \cite[Proposition 2.5]{CDSjmaa} we know that
$\{\mathfrak{A}_{k}(E)\}_{k}$ is a coherent sequence, hence it is a holomorphy type by Remark \ref{rema}. As to condition \ref{pi-tipo de holomorfia}(a2), \cite[Lemma 2.1]{CDSjmaa} shows that
$$\overline{\mathcal{P}_{f}(^{k}E)}^{\mathfrak{A}_{k}}=\mathfrak{A}_{k}(E) ~{\rm for~every~}k.$$
So all that is left to check is the inequality in condition (a1) of Definition \ref{pi-tipo de holomorfia}. By assumption we know that
$\|T\|_{\mathfrak{A}_{k}(E)^{\prime}}=\|B_k(T)\|_{\mathfrak{B}_{k}(E^{\prime})}$
for every $T\in\ \mathfrak{A}_{k}(E)^{\prime}$. Let $C$ be the constant of condition \ref{coherent sequence}(a) for the coherent sequence $\{\mathfrak{B}_{k}(E^{\prime})\}_{k}$. By the inequality in Remark \ref{rema},
$$\|B_{k}(T)\| \leq C^k  \|B_{k}(T)\|_{\mathfrak{B}_{k}(E^{\prime})}, $$
for all $T\in\ \mathfrak{A}_{k}(E)^{\prime}$ and $k \in \mathbb{N}_0$. Thus,
\begin{align*}
\|\phi^{k}\|_{\mathfrak{A}_{k}\left(
E\right)}& =\sup\limits_{\|T\|_{\mathfrak{A}_{k}(E)^{\prime}}= 1}
|T(\phi^{k})| = \sup\limits_{\|T\|_{\mathfrak{A}_{k}(E)^{\prime}}=
1} |B_{k}(T)(\phi)| \\&\leq \|\phi\|^{k} \cdot
\sup\limits_{\|T\|_{\mathfrak{A}_{k}(E)^{\prime}}= 1}
\|B_{k}(T)\| \\&
\leq \|\phi\|^{k} \cdot C^k \cdot \sup\limits_{\|T\|_{\mathfrak{A}_{k}(E)^{\prime}}= 1} \|B_{k}(T)\|_{\mathfrak{B}_{k}(E^{\prime})}\\&=
 \|\phi\|^{k} \cdot C^k \cdot \sup\limits_{\|T\|_{\mathfrak{A}_{k}(E)^{\prime}}= 1} \|T\|_{\mathfrak{A}_{k}(E)^{\prime}}=C^{k}\cdot \|\phi\|^{k},
\end{align*}
for all $\phi\in E^{\prime}$ and $k\in\mathbb{N}_0$.
\end{proof}

To continue we need the following result:

\begin{proposition}{\rm \cite[Lemma 3.3]{CDSjmaa}}
\label{lemma_carando}Let $\left\{  \mathfrak{B}_{k}\left(
E^{\prime}\right) \right\}  _{k}$ be a coherent multiplicative
sequence and $\{\mathfrak{A}_{k}(E)\}_{k}$ be such that
$\mathfrak{A}_{k}\left(  E\right) ^{\prime}=\mathfrak{B}_{k}\left(
E^{\prime}\right) $ for every $k.$ Let  $k\geq l$, $P\in\mathfrak{A}_{k}\left(
E\right) $ and $\varphi\in\mathfrak{A}_{k-l}\left(  E\right)
^{\prime}$ be given. Then the $l$-homogeneous polynomial
$x \in E \mapsto\varphi\left(
P_{x^{l}}\right) \in \mathbb{C}  $ belongs to $\mathfrak{A}_{l}\left(  E\right)  $ and%
\[
\Vert x\mapsto\varphi\left(  P_{x^{l}}\right)
\Vert_{\mathfrak{A}_{l}\left( E\right)  }\leq
M^{k}\Vert\varphi\Vert_{\mathfrak{A}_{k-l}\left(  E\right)
^{\prime}}\Vert P\Vert_{\mathfrak{A}_{k}\left(  E\right)  }.
\]

\end{proposition}

\begin{proposition}
\label{coerente_pi2}Let $\left\{  \mathfrak{B}_{k}\left(  E^{\prime}\right)  \right\}
_{k}$ be a coherent multiplicative sequence and
$\{\mathfrak{A}_{k}(E)\}_{k}$ be such that
$\mathfrak{A}_{k}(E)^{\prime }=\mathfrak{B}_{k}(E^{\prime})$ for
every $k$. Then $\mathfrak{A}\left( E\right) =\left\{
\mathfrak{A}_{k}\left( E\right)  \right\} _{k}$ is a
$\pi_{2}$-holomorphy type from $E$ to $\mathbb{C}$.
\end{proposition}

\begin{proof} Again by \cite[Proposition 2.5]{CDSjmaa} we get that $\{\mathfrak{A}_{k}(E)\}_{k}$ is a coherent sequence, so the space $ \mathcal{H}_{b\mathfrak{A}}(E)$ is well defined.
Let $T\in\left[  \mathcal{H}_{b\mathfrak{A}}(E)\right]
^{\prime}$ and
$m,k\in\mathbb{N}_{0}$ with $k\leq m$ be given. Note that%
\[
T\left(  \widehat{A(\cdot)^{k}}\right)  =\left(  x\mapsto T\left(  P_{x^{m-k}%
}\right)  \right)
\]
for every $P\in\mathfrak{A}_{m}\left(  E\right),$ where $A$ is the 
$m$-linear symmetric mappings on $E^{m}$ such that
$P=\hat{A}$. Therefore from Proposition \ref{lemma_carando}
it follows that $T\left(  \widehat
{A(\cdot)^{k}}\right)  $ belongs to $\mathfrak{A}_{m-k}\left(  E\right)  $ and%
\[
\Vert T(\widehat{A(\cdot)^{k}})\Vert_{\mathfrak{A}_{m-k}\left(
E\right) }=\Vert x\mapsto T\left(  P_{x^{m-k}}\right)
\Vert_{\mathfrak{A}_{m-k}\left(
E\right)  }\leq M^{m} \cdot \Vert T|_{\mathfrak{A}_{k}\left(  E\right)  }%
\Vert_{\mathfrak{A}_{k}\left(  E\right)  ^{\prime}}\cdot\Vert P\Vert_{\mathfrak{A}%
_{m}\left(  E\right)  },
\]
where $T|_{\mathfrak{A}_{k}\left(  E\right)}$ obviously means the restriction of $T$ to $\mathfrak{A}_{k}\left(  E\right)$. Since $T\in\left[  \mathcal{H}_{b\mathfrak{A}}(E)\right]
^{\prime},$ there
are $C>0$ and $\rho>0$ such that%
\[
\left\vert T\left(  f\right)  \right\vert \leq C \cdot p_{\rho}\left(  f\right)
\]
for every $f\in\mathcal{H}_{b\mathfrak{A}}(E).$ In particular,
\[
\left\vert T\left(  Q\right)  \right\vert \leq C \cdot p_{\rho}\left(
Q\right) =C \cdot \rho^{k} \cdot \Vert Q\Vert_{\mathfrak{A}_{k}\left(  E\right)
}
\]
for every $Q\in\mathfrak{A}_{k}\left(  E\right)  $, so
\[
\Vert T|_{\mathfrak{A}_{k}\left(  E\right)
}\Vert_{\mathfrak{A}_{k}\left( E\right)  ^{\prime}}\leq C\cdot \rho^{k}.
\]
Therefore,%
\[
\Vert T(\widehat{A(\cdot)^{k}})\Vert_{\mathfrak{A}_{m-k}\left(
E\right)
}\leq M^{m}\cdot \Vert T|_{\mathfrak{A}_{k}\left(  E\right)  }\Vert_{\mathfrak{A}%
_{k}\left(  E\right)  ^{\prime}}\cdot \Vert
P\Vert_{\mathfrak{A}_{m}\left( E\right)  }\leq C\cdot M^{m} \cdot \rho^{k}\cdot \Vert
P\Vert_{\mathfrak{A}_{m}\left(  E\right) },
\]
which completes the proof.
\end{proof}

A combination of Proposition \ref{coerente_pi1} with Proposition \ref{coerente_pi2} makes clear that Corollary \ref{cdm2} is a particular case of Theorem \ref{main2}:

\begin{corollary}Let $\left\{  \mathfrak{B}_{k}\left(  E^{\prime}\right)
\right\} _{k}$ be a coherent multiplicative sequence and
$\{\mathfrak{A}_{k}(E)\}_{k}$ be such that
$\mathfrak{A}_{k}(E)^{\prime }=\mathfrak{B}_{k}(E^{\prime})$ for
every $k$. Then $\mathfrak{A}\left( E\right) =\left\{
\mathfrak{A}_{k}\left( E\right)  \right\} _{k}$ is a
$\pi_1$-$\pi_{2}$-holomorphy type.
\end{corollary}

Now we prove that our results are more than formal generalizations of the known results, in the sense that there are concrete cases covered by our results and not covered by the results of \cite{CDSjmaa}. Of course it is enough to give examples of $\{\mathfrak{A}_{k}(E)\}_{k}$ such that:\\
(i) $\{\mathfrak{A}_{k}(E)\}_{k}$ is a $\pi_1$-$\pi_{2}$-holomorphy type,\\
(ii) $\mathfrak{A}_{k}(E)^{\prime }=\mathfrak{B}_{k}(E^{\prime})$ for every $k$,\\ (iii) $\left\{  \mathfrak{B}_{k}\left(  E^{\prime}\right)\right\} _{k}$ fails to be a coherent sequence.

\begin{example}\label{exemplos}\rm (a) Consider the space $\mathcal{P}_{\left( p,m\left(
s;q\right) \right)}\left(^{m}E\right)$ of all absolutely $\left( p,m\left(
s;q\right) \right)$-summing $m$-homogeneous
polynomials on $E$ introduced by Matos \cite[Section 3]{matosjmaa}, where $0 < q \leq s \leq +\infty$ and $p \geq q$. In general $\left(\mathcal{P}_{\left( p,m\left(
s;q\right) \right)}\left(^{m}E\right)\right)_{m \in \mathbb{N}_0}$ is not a holomorphy type, hence fails to be a coherent sequence. For example, making $s=q = p > 1$, the space $\mathcal{P}_{\left( p,m\left(
p;p\right) \right)}\left(^{m}E\right)$ coincides with the space of absolutely $p$-summing $m$-homogeneous polynomials (see \cite[p. 843]{matosjmaa}), which is not a holomorphy type by \cite[Example 3.2]{monat}.

On the other hand, Matos proved in \cite[Section 8.2]{matos7} that if
$E^{\prime}$ has the bounded approximation property,
then the Borel transform $\mathcal{B}_{\widetilde{N},\left(
s;\left( r,q\right) \right) }$ establishes an isometric
isomorphism between $\left[\mathcal{P}_{\widetilde{N},\left(
s;\left( r,q\right) \right)}\left(^{m}E\right)\right]^{\prime}$
and $\mathcal{P}_{\left( s^{\prime},m\left(
r^{\prime};q^{\prime}\right) \right)}\left(^{m}E^{\prime}\right)$,
where $\mathcal{P}_{\widetilde{N},\left( s;\left( r,q\right)
\right)}\left(^{m}E\right)$ denotes the space of all $\left(
s;\left( r,q\right) \right)$-quasi-nuclear $m$-homogeneous
polynomials on $E$ (as usual $s', r', q'$ denote the conjugates of $s, r, q$, respectively). 
 So
\begin{equation}\label{duality}\left[\mathcal{P}_{\widetilde{N},\left(
s;\left( r,q\right) \right)}\left(^{m}E\right)\right]^{\prime}=
\mathcal{P}_{\left( s^{\prime},m\left(
r^{\prime};q^{\prime}\right) \right)}\left(^{m}E^{\prime}\right) {\rm ~for~every~} m.
\end{equation}
The proof that $\mathcal{P}_{\widetilde{N},\left(
s;\left( r,q\right) \right)}\left(^{m}E\right)$ is a $\pi_1$-holomorphy type can be found in \cite[Sections 8.2 and 8.3]{matos7} and that it is a $\pi_2$-holomorphy type in \cite[Proposition 9.1.5]{matos7}.

\medskip

\noindent (b) X. Mujica proved in \cite[Teorema 2.5.1]{ximena} that if $E^{\prime}$ has the
bounded approximation property, $p \geq 1$ and $F$ is reflexive,
then the Borel transform $\mathcal{B}_{\sigma(p)}$ establishes an
isometric isomorphism between
$\left[\mathcal{P}_{\sigma(p)}\left(^{m}E;F\right)\right]^{\prime}$
and $\mathcal{P}_{\tau(p)}\left(^{m}E^{\prime};F^{\prime}\right)$,
where $\mathcal{P}_{\sigma(p)}\left(^{m}E;F\right)$ denotes the
space of all $\sigma(p)$-nuclear $m$-homogeneous polynomials from
$E$ into $F$, and
$\mathcal{P}_{\tau(p)}\left(^{m}E^{\prime};F^{\prime}\right)$
denotes the space of all $\tau(p)$-summing $m$-homogeneous
polynomials from $E^{\prime}$ into $F^{\prime}$. Making $F =\mathbb{C}$ we get
$$\left[\mathcal{P}_{\sigma(p)}\left(^{m}E\right)\right]^{\prime} = \mathcal{P}_{\tau(p)}\left(^{m}E^{\prime}\right) {\rm for ~every~} m.$$
Again, and for the same reason, $\left(\mathcal{P}_{\tau(p)}\left(^{m}E^{\prime}\right)\right)_{m=0}^\infty$ is not a holomorphy type in general, consequently it fails to be a coherent sequence. Condition (a1) of Definition \ref{pi-tipo de holomorfia} follows easily because $\mathcal{P}_{\sigma(p)}$ is a polynomial ideal. Condition (a2) is proved in \cite[Proposi\c c\~ao 2.4.4]{ximena}, so $\mathcal{P}_{\sigma(p)}\left(^{m}E\right)$ is a $\pi_1$-holomorphy type. The fact that $\mathcal{P}_{\sigma(p)}\left(^{m}E\right)$ is a $\pi_2$-holomorphy type is proved in \cite[Lema 3.2.6]{ximena} with $K =1$.
\end{example}

\section{Proofs of the main results}

The first step is the definition of the Borel transform. A holomorphy type from $E$ to $F$ shall be denoted by either $\Theta$ or $( \mathcal{P}_{\Theta}(^{m}E;F))_{m=0}^\infty$.

\begin{definition}\rm (a) Let $\Theta$ be a $\pi_{1}$-holomorphy type from $E$ to
$F$. It is clear that the {\it Borel transform}
\[
\mathcal{B}_{\Theta}\colon\left[  \mathcal{P}_{\Theta}(^{m}E;F)\right]
^{\prime}\longrightarrow\mathcal{P}(^{m}E^{\prime};F^{\prime})~, ~ \mathcal{B}_{\Theta} T(\phi)(y)=T(\phi^{m}y),\] for $T\in\left[
\mathcal{P}_{\Theta}(^{m}E;F)\right]  ^{\prime}$, $\phi\in E^{\prime}$ and
$y\in F$, is well defined and linear. Moreover, ${\cal B}_\Theta$ is continuous and injective by conditions (a1) and (a2) of Definition \ref{pi-tipo de holomorfia}. So, denoting the range of $\mathcal{B}_{\Theta}$ in $\mathcal{P}%
(^{m}E^{\prime};F^{\prime})$ by $\mathcal{P}_{\Theta^{\prime}}(^{m}E^{\prime
};F^{\prime})$, the correspondence
$$\mathcal{B}_{\Theta}
T \in \mathcal{P}%
_{\Theta^{\prime}}(^{m}E^{\prime};F^{\prime})\mapsto \|\mathcal{B}_{\Theta}
T\|_{\Theta^{\prime}}:=\|T\|,  $$
defines a norm on $\mathcal{P}%
_{\Theta^{\prime}}(^{m}E^{\prime};F^{\prime})$. In this fashion the spaces $\left(  \left[  \mathcal{P}_{\Theta}%
(^{m}E;F)\right]  ^{\prime}\;,\|\cdot\|\right)  $ and $(\mathcal{P}_{\Theta^{\prime}}(^{m}E^{\prime};F^{\prime}),\;\|\cdot
\|_{\Theta^{\prime}})$ are isometrically isomorphic.

\medskip

\noindent (b) Let $(\mathcal{P}_{\Theta}(^{m}%
E))_{m=0}^{\infty}$ be a $\pi_{1}$-holomorphy type from $E$ to $\mathbb{C}$.
A holomorphic function $f\in\mathcal{H}(E^{\prime})$ is said to be {\it of $\Theta^{\prime}$-exponential type} if\\
(b1) $\hat{d}^{m}f(0)\in\mathcal{P}_{\Theta^{\prime}%
}(^{m}E^{\prime})$ for every $m\in\mathbb{N}_{0}$;\\
(b2) There are constants $C\geq0$ and
$c>0$ such that%
\[
\Vert{\hat{d}}^{m}f(0)\Vert_{\Theta^{\prime}}\leq Cc^{m},
\]
for all $m\in\mathbb{N}_{0}$.

The vector space of all such functions
is denoted by ${\rm Exp}_{\Theta^{\prime}}(E^{\prime})$.
\end{definition}

 The change we made in the definition of $\pi_1$-holomorphy types does not affect the validity of \cite[Corollary 2.1]{favaro-jatoba1}. So if $\Theta$ is a $\pi_1$-holomorphy type from $E$ to $\mathbb{C}$, then all nuclear entire functions of bounded type belong to $\mathcal{H}_{\Theta b}(E)$. In particular, the functions of the form $e^\phi$, for $\phi \in E'$, belong to $\mathcal{H}_{\Theta b}(E)$. The proof of \cite[Theorem 2.1]{favaro-jatoba1} is not affected either:

\begin{proposition}{\rm \cite[Theorem 2.1]{favaro-jatoba1}}
\label{Propo isomorfismo Hthetab} If $\Theta$ is a $\pi_{1}$-holomorphy type
from $E$ to $\mathbb{C}$, then the Borel transform%
\[
\mathcal{B}\colon\lbrack\mathcal{H}_{\Theta b}(E)]^{\prime}\longrightarrow
{\rm Exp}_{\Theta^{\prime}}(E^{\prime})~,~\mathcal{B}T(\phi)=T(e^{\phi}),\]
 for all $T\in\lbrack\mathcal{H}%
_{\Theta b}(E)]^{\prime}$ and $\phi\in E^{\prime},$ is an algebraic
isomorphism.
\end{proposition}

\begin{proposition}\label{densidade de ephi b}
Let $\Theta$ be a $\pi_{1}$-holomorphy type from $E$ to $\mathbb{C}$ and $U$ be a
non-empty open subset of $E^{\prime}$. Then the set
$$S={\rm span}\{e^{\phi}:\phi\in U \}$$

\noindent is dense in $\mathcal{H}_{\Theta b}(E)$.
\end{proposition}

\begin{proof}
Assume that $S$ is not dense in $\mathcal{H}_{\Theta b}(E)$. In this case, the geometric Hahn-Banach Theorem gives a nonzero functional $T\in [\mathcal{H}_{\Theta b}(E),
\tau_{\Theta}]^{\prime}$ that vanishes on $\overline{S}$.  In particular $T(e^{\phi})=0$ for each
$\phi\in U$. So $\mathcal{B}T(\phi)=T(e^{\phi})=0$ for every $\phi \in U$.
Thus $\mathcal{B} T$ is a holomorphic function that vanishes on the open non-void set $U$. It follows that $\mathcal{B} T\equiv 0$ on $E'$. Since
$\mathcal{B}$ is injective by Proposition \ref{Propo isomorfismo
Hthetab}, $T\equiv 0$. This contradiction proves that $S$ is dense in $\mathcal{H}_{\Theta b}(E)$.
\end{proof}

\bigskip

Let $\Theta$ be a holomorphy type from $E$ to $\mathbb{C}$. The linear space of all convolution operators on $\mathcal{H}_{\Theta b}(E)$ is denoted by $\mathcal{O}(\mathcal{H}_{\Theta b}(E)).$  We define the map $\Gamma_{\Theta}$ by%
\begin{align*}
\Gamma_{\Theta}  &  \colon\mathcal{O}(\mathcal{H}_{\Theta b}(E))\longrightarrow
\lbrack\mathcal{H}_{\Theta b}(E)]^{\prime}\\
&  \;\;\;\:\:\:\:\;\;\;\;\;\;\;\;\;\;\;\;L\mapsto\Gamma_{\Theta}(L)\colon
\mathcal{H}_{\Theta b}(E)\longrightarrow\;\;\;\mathbb{K}\\
&
\;\;\;\;\;\;\;\;\;\;\;\;\;\;\;\;\:\:\:\:\:\:\;\;\;\;\;\;\;\;\;\;\;\;\;\;\;\;\;\;\;\;\;
\;\;\;\;\;\;f\;\;\mapsto
\;\;\Gamma_{\Theta}(L)(f):=(L\left(  f\right)  )(0).
\end{align*}
Remember the definition of $\delta_0$ to see that $\Gamma_\Theta (L) = \delta_0 \circ L$. It is clear that $\Gamma_{\Theta}$ is a well defined linear map.

\begin{lemma}
\label{L identidade ss B não constante} Let $\Theta$ be a $\pi_{1}$-holomorphy
type from $E$ to $\mathbb{C}$ and $L\in\mathcal{O}(\mathcal{H}_{\Theta b}(E))$ be given. Then:
\begin{enumerate}
\item[{\rm (a)}] $L(e^{\phi})=\mathcal{B}(\Gamma_{\Theta}(L))(\phi)\cdot e^{\phi
}$ for every $\phi\in E^{\prime}.$

\item[{\rm(b)}] $L$ is a scalar multiple of the identity if and only if
$\mathcal{B}(\Gamma_{\Theta}(L))$ is constant.
\end{enumerate}
\end{lemma}

\begin{proof}
(a) Since $\Gamma_{\Theta}(L)\in\lbrack\mathcal{H}_{\Theta b}(E)]^{\prime},$
from Theorem \ref{Propo isomorfismo Hthetab} we know that%
\[
\mathcal{B}(\Gamma_{\Theta}(L))(\phi)=\Gamma_{\Theta}(L)(e^{\phi
})=L(e^{\phi})(0)
\]
for each
$\phi\in E^{\prime}.$ Therefore%
\begin{align*}
L(e^{\phi})(y) &  =[\tau_{-y}(L(e^{\phi}))](0)\\
&  =[L\left(  \tau_{-y}(e^{\phi})\right)  ](0)\\
&  =[L\left(  e^{\phi(y)}\cdot e^{\phi}\right)  ](0)\\
&  =e^{\phi(y)}\cdot L(e^{\phi})(0)\\
&  =e^{\phi(y)}\cdot \mathcal{B}(\Gamma_{\Theta}(L))(\phi)\\
& = \left(\mathcal{B}(\Gamma_{\Theta}(L))(\phi) \cdot e^\phi
\right)(y),
\end{align*}

for all $y\in E.$\newline(b) Let $\lambda\in\mathbb{C}$ be such that
$\mathcal{B}(\Gamma_{\Theta}(L))(\phi)=\lambda$ for every $\phi\in
E^{\prime}.$ By (a) it follows that%
\[
L(e^{\phi})=\mathcal{B}(\Gamma_{\Theta}(L))(\phi)\cdot e^{\phi}=\lambda
e^{\phi}
\]
for every $\phi\in
E^{\prime}.$ The continuity of $L$ and the denseness of $\{e^\phi : \phi \in E'\}$ in $\mathcal{H}_{\Theta b}(E)$ (Proposition \ref{densidade de ephi b}) yield
that $L(f)=\lambda f$ for every $f\in\mathcal{H}_{\Theta b}(E)$.

Conversely, let $\lambda\in\mathbb{C}$ be such that $L(f)=\lambda f$ for every
$f\in\mathcal{H}_{\Theta b}(E)$. Calling on (a) again we get%

\[
\lambda e^{\phi}=L(e^{\phi})=\mathcal{B}(\Gamma_{\Theta}%
(L))(\phi)\cdot e^{\phi},
\]
hence $\mathcal{B}(\Gamma_{\Theta}(L))(\phi)=\lambda$ for every
$\phi\in E^{\prime}.$
\end{proof}

In the proof of our main result we shall use the following
criterion, which was obtained, independently, by Kitai
\cite{kitai} and Gethner and Shapiro \cite{gethner}:

\begin{theorem}
\label{Hypercyclity criterion} {\rm(Hypercyclicity Criterion)} Let
$X$ be a separable Fr\'echet space. A
continuous linear operator $T \colon X
\longrightarrow X$ is hypercyclic if there are dense subsets
$Z\mathrm{,} \,Y\subseteq X$ and a map $S\colon Y \longrightarrow
Y$ satisfying the following three conditions:

\item {\rm(a)} For each $z\in Z$, $T^{n}(z) \longrightarrow0$ when
$n\longrightarrow\infty$;

\item {\rm(b)} For each $y\in Y$, $S^{n}(y) \longrightarrow0$ when
$n\longrightarrow\infty$;

\item {\rm(c)} $T\circ S= I_{Y}$.
\end{theorem}

The last ingredient we need to give the proof of Theorem \ref{main1} is the next result.

\begin{proposition}\label{Linearmente indep}
Let $( \mathcal{P}_{\Theta}(^{m}E))_{m=0}^\infty$ be a
$\pi_{1}$-holomorphy type from $E$ to $\mathbb{C}$. Then the set
$$B=\{e^\phi : \phi\in E^{\prime}\}$$

\noindent is a linearly independent subset of $\mathcal{H}_{\Theta
b}(E)$. 
\end{proposition}

\begin{proof} We have already remarked that $\{e^\phi : \phi\in E^{\prime}\} \subseteq \mathcal{H}_{\Theta b}(E)$ whenever
$\Theta$ is a $\pi_1$-holomorphy type. Given $a \in E$, from
\cite[Propostion 3.1(i)]{favaro-jatoba1} we know that the
differentiation operator
$$D_a\colon\mathcal{H}_{\Theta
b}(E)\longrightarrow\mathcal{H}_{\Theta b}(E)~,~D_a\left(
f\right) =df\left(  \cdot\right)  \left(  a\right)$$
is well defined. Now one just has to follow the lines of the proof of \cite[Lemma 2.3]{aron} to get the result. 
\end{proof}

\noindent\textbf{Proof of Theorem \ref{main1}}.
Let $L \colon\mathcal{H}_{\Theta b}(E) \longrightarrow\mathcal{H}_{\Theta
b}(E)$ be a convolution operator which is not a scalar multiple of the
identity. We shall show that $L$ satisfies the Hypercyclicity Criterion of Theorem \ref{Hypercyclity criterion}. First of all, since $E'$ is separable and $\Theta$ is a $\pi_1$-holomorphy type, we have that $\mathcal{H}_{\Theta b}(E)$ is separable as well. We have already remarked that $\mathcal{H}_{\Theta b}(E)$ is a Fr\'echet space. By $\Delta$ we mean the open unit disk in $\mathbb{C}$. Consider the sets%

\[
V=\{\phi\in E^{\prime}: |\mathcal{B}(\Gamma_{\Theta}(L))(\phi)| <1 \}=\mathcal{B}(\Gamma_{\Theta}(L))^{-1}(\Delta)
\]

\noindent and%

\[
W=\{\phi\in E^{\prime}: |\mathcal{B}(\Gamma_{\Theta}(L))(\phi)| >1 \}=\mathcal{B}(\Gamma_{\Theta}(L))^{-1}(\mathbb{C}-\overline{\Delta}).
\]
Since $L$ is not a scalar multiple of the identity, Lemma
\ref{L identidade ss B não constante}(b) yields that $\mathcal{B}(\Gamma_{\Theta}(L))$ is non
constant. Therefore, it follows from Liouville's Theorem that $V$ and $W$ are non-empty open subsets of $E^{\prime}$. Consider now the following subspaces of  $\mathcal{H}_{\Theta
b}(E)$:
\[
H_V = {\rm span}\{e^{\phi}: \phi\in V\} {\rm ~~and~~}H_W ={\rm span}\{e^{\phi}: \phi\in W\}.
\]%
By Proposition \ref{densidade de ephi b} we know that both $H_V$ and $H_W$ are dense in $\mathcal{H}_{\Theta
b}(E)$.

Let us deal with $H_V$ first. Given $\phi \in V$, from Lemma \ref{L identidade ss B não constante}(a) we have
\[
L(e^{\phi})=\mathcal{B}(\Gamma_{\Theta}(L))(\phi)\cdot e^{\phi} \in H_V.
\]
So $L(H_V) \subseteq H_V$ because $L$ is linear. 
Applying Lemma \ref{L identidade ss B não constante}(a) and the linearity of $L$ once again we get
\[
L^{n}(e^{\phi})=[\mathcal{B}(\Gamma_{\Theta}(L))(\phi)]^{n}\cdot e^{\phi}
\]
for all $n\in\mathbb{N}$ and $\phi \in V$. Consequently,
\[
L^{n}(f)=[\mathcal{B}(\Gamma_{\Theta}(L))(\phi)]^{n}\cdot f
\]
for all $n\in\mathbb{N}$ and $f\in H_V$. Since $|\mathcal{B}(\Gamma_{\Theta}(L))(\phi)|< 1$ whenever $\phi\in V$, it follows that $L^{n}(f)\longrightarrow 0$ when
$n\longrightarrow\infty$ for each $f\in H_V$.

Now we handle $H_W$. For each $\phi\in W$, $\mathcal{B}(\Gamma_{\Theta}(L))(\phi) \neq 0$, so we can define
\[
S(e^{\phi}):=\dfrac{e^{\phi}}{\mathcal{B}(\Gamma_{\Theta}(L))(\phi)} \in \mathcal{H}_{\Theta b}(E).
\]
By Proposition \ref{Linearmente indep},
$\{e^{\phi}: \phi\in W\}$ is a linearly independent set, so we can extend $S$ to
$H_W$ by linearity. Therefore $S(H_W)\subseteq H_W$ and
\[
S^{n}(f)=\dfrac{f}{[\mathcal{B}(\Gamma_{\Theta}(L))(\phi)]^{n}}
\]
for all $f\in H_W$ and $n\in\mathbb{N}$. Since $|\mathcal{B}(\Gamma_{\Theta}(L))(\phi)|> 1$ whenever $\phi\in W$, it follows that $S^{n}(f)\longrightarrow0$ when
$n\longrightarrow\infty$ for each $f\in H_W$.

Finally, $L\circ S(f)=f$ for every $f\in
H_W$, so $L$ is hypercyclic.\hfill$\Box$



\bigskip

Let us proceed to the proof of Theorem \ref{main2}. The next result is needed. It is an adaptation of \cite[Theorem 3.1]{favaro-jatoba1} to the new definition of $\pi_2$-holomorphy type. In this case it is worth giving the details.

\begin{proposition}
\label{Teorema T* convolucao} If $(\mathcal{P}_{\Theta}(^{m}E))_{m=0}^{\infty
}$ is a $\pi_{2}$-holomorphy type from $E$ to $\mathbb{C}$, $T\in\lbrack\mathcal{H}_{\Theta
b}(E)]^{\prime}$ and $f\in\mathcal{H}_{\Theta b}(E),$ then $T\ast
f\in\mathcal{H}_{\Theta b}(E)$ and the mapping $T\ast$ defines a convolution
operator on $\mathcal{H}_{\Theta b}(E)$.
\end{proposition}

\begin{proof}
Since $T\in\left[\mathcal{H}_{\Theta b}(E)\right]^{\prime}$, there are constants $C\geq0$ and $\rho>0$ such that
\begin{equation*}\label{desigT}
\left\vert T\left(  f\right)  \right\vert \leq C\left\Vert
f\right\Vert _{\Theta,\rho}
\end{equation*}
for all $f\in\mathcal{H}_{\Theta b}(E)$.
By \cite[Proposition 3.1]{favaro-jatoba1},
\begin{align}
(T\ast f)(x)& =T(\tau_{-x}f)=T\left(  \sum_{m=0}^{\infty}\frac{1}{m!}\hat{d}%
^{m}f(x)\right) \nonumber\\
&=\sum_{m=0}^{\infty}\frac{1}{m!}\sum_{k=0}^{\infty}\frac{1}{k!}%
T(\Hhat{d^{k+m}f(0)(\cdot)^k})(x) \label{des3.5}%
\end{align}
for every $x \in E$. By Definition \ref{pi-tipo de holomorfia}(b) there is a constant $K$ such that
$$T\left(\Hhat{d^{k+m}f(0)(\cdot)^k}\right)\in
\mathcal{P}_{\Theta}(^{m}E) {\rm ~~  and}$$%
\[
\left\Vert T\left(\Hhat{d^{k+m}f(0)(\cdot)^k}\right)\right\Vert _{\Theta}\leq
CK^{m+k}\rho ^{k}\left\Vert \hat{d}^{m+k}f(0)\right\Vert _{\Theta}
\]
for all $k,m \in \mathbb{N}_0$. For $\rho_{0}>\rho$ we can write%
\begin{align*}
\left\Vert \sum_{k=0}^{\infty}\frac{1}{k!}T\left(  \overset
{\begin{picture}(60,5)\put(0,0){\line(6,1){30}}\put(30,5){\line(6,-1){30}}\end{picture}}%
{d^{k+m}f(0)(\cdot)^{k}}\right)  \right\Vert _{\Theta}& \leq \sum_{k=0}^{\infty
}\frac{1}{k!}\left\Vert T\left(  \overset
{\begin{picture}(60,5)\put(0,0){\line(6,1){30}}\put(30,5){\line(6,-1){30}}\end{picture}}%
{d^{k+m}f(0)(\cdot)^{k}}\right)  \right\Vert _{\Theta}\\
&\leq\sum_{k=0}^{\infty}\frac{1}{k!}CK^{m+k}\rho^{k}\left\Vert \hat{d}%
^{m+k}f(0)\right\Vert _{\Theta}\\
&\leq\sum_{k=0}^{\infty}\frac{1}{k!}CK^{m+k}%
\rho_{0}^{k}\left\Vert \hat{d}^{m+k}f(0)\right\Vert _{\Theta}\\
&=C\frac{m!}{\rho_{0}^{m}}\sum_{k=0}^{\infty}\frac{\left(  m+k\right)  !}%
{m!k!}\cdot\frac{K^{m+k}}{\left(  m+k\right)  !}\rho_{0}^{m+k}\left\Vert
\hat{d}^{m+k}f(0)\right\Vert _{\Theta}\\
&\leq C\frac{m!}{\rho_{0}^{m}}\sum_{k=0}^{\infty}\frac{(2K)^{m+k}}{(m+k)!}%
\rho_{0}^{m+k}\left\Vert \hat{d}^{m+k}f(0)\right\Vert _{\Theta}\\
&=C\frac{m!}{\rho_{0}^{m}}\left\Vert \sum_{k=m}^{\infty}\frac{1}{k!}\hat{d}%
^{k}f(0)\right\Vert _{\Theta,2K\rho_{0}}\leq C\frac{m!}{\rho_{0}^{m}}\Vert
f\Vert_{\Theta,2K\rho_{0}}<\infty.
\end{align*}
This means that
\[
P_{m}=\sum_{k=0}^{\infty}\frac{1}{k!}T\left(
\Hhat{d^{k+m}f(0)(\cdot)^k}\right)
\]
belongs to $\mathcal{P}_{\Theta}(^{m}E)$ and%
\begin{eqnarray}
\Vert P_{m}\Vert_{\Theta}\leq C\frac{m!}{\rho_{0}^{m}}\Vert f\Vert
_{\Theta,2K\rho_{0}}.\label{des3.7}
\end{eqnarray}
Hence%
\[
\lim_{m\rightarrow\infty}\left(  \frac{1}{m!}\Vert
P_{m}\Vert_{\Theta}\right) ^{\frac{1}{m}}\leq\frac{1}{\rho_{0}}
\]
for every $\rho_{0}>\rho.$ This implies that%
\[
\lim_{m\rightarrow\infty}\left(  \frac{1}{m!}\Vert
P_{m}\Vert_{\Theta}\right) ^{\frac{1}{m}}=0.
\]
Therefore, it follows from $(\ref{des3.5})$ that $(T\ast
f)=\displaystyle\sum_{m=0}^{\infty}\frac{1}{m!}P_{m} \in
\mathcal{H}_{\Theta b}(E).$ It is clear that
$T\ast$ is
linear. For $\rho_{1}>0$, from $(\ref{des3.7})$ we get%
\begin{align*}
\Vert T\ast f\Vert_{\Theta,\rho_{1}}&=\sum_{m=0}^{\infty}\frac{\rho_{1}^{m}%
}{m!}\Vert P_{m}\Vert_{\Theta}\\
&\leq\sum_{m=0}^{\infty}\frac{C\rho_{1}^{m}}%
{m!}\frac{m!}{(\rho_{1}+\rho)^{m}}\Vert f\Vert_{\Theta,2K(\rho_{1}+\rho)}\\
&= \left(  \sum_{m=0}^{\infty}\frac{C\rho_{1}^{m}}{(\rho_{1}+\rho)^{m}%
}\right)  \Vert f\Vert_{\Theta,2K(\rho_{1}+\rho)},
\end{align*}
proving that $T\ast$ is continuous. Now we have%
\begin{align*}
(T\ast\tau_{a}f)(x)&=T(\tau_{-x}\circ\tau_{a}f)=T(\tau_{-x+a}f)\\
&=(T\ast f)(-(-x+a))\\&=(T\ast f)(x-a)=\tau_{a}(T\ast f)(x),
\end{align*}
for all $x,a\in E$. This completes the proof that $T\ast$ is a
convolution operator.
\end{proof}

\noindent\textbf{Proof of Theorem \ref{main2}}. The operator $\bar{\Gamma}_{\Theta}(T)$ is a convolution operator for each
$T\in[\mathcal{H}_{\Theta b}(E)]^{\prime}$ by Proposition \ref{Teorema T* convolucao}. Suppose that there
is $\lambda\in\mathbb{C}$ such that $\bar{\Gamma}_{\Theta}(T)(f)=\lambda\cdot
f$ for all $f\in\mathcal{H}_{\Theta b}(E)$. Then
\[
\lambda\cdot f(x)=\bar{\Gamma}_{\Theta}(T)(f)(x)=(T\ast f)(x)=T(\tau
_{-x}f)
\]
for every $x \in E$. In particular,
\[
\lambda\cdot\delta_{0}(f)=\lambda\cdot f(0)=T(\tau_{0}f)=T(f)
\]
for every $f\in\mathcal{H}_{\Theta b}(E)$. Hence $T=\lambda\cdot\delta_{0}$. This contradiction shows that $\bar{\Gamma}_{\Theta}(T)$ is not a scalar multiple of the identity, hence hypercyclic by Theorem \ref{main1}.\hfill$\Box$

\section{Further results}
In this section we show that several related results that appear in the literature have analogues in the context of $\pi_j$-holomorphy types, $j = 1,2$.

We start with an analogue of \cite[Corollary 8]{aronmarkose}:

\begin{proposition}
\label{dense_range}If $E^{\prime}$ is separable and $(\mathcal{P}_{\Theta}(^{n}E))_{n=0}^{\infty
}$ is a $\pi_{1}$-holomorphy type from $E$ to $\mathbb{C}$, then every nonzero convolution
operator on $\mathcal{H}_{\Theta b}(E)$ has dense range.
\end{proposition}

\begin{proof}
Let $L\neq0$ be a convolution operator. If $L$ is a scalar multiple of the identity, then clearly $L$ is
surjective. Suppose now that $L$ is not a scalar multiple of the identity. By Proposition \ref{densidade de ephi b}, ${\rm span}\{e^{\phi}: \phi\in E^{\prime
}\}$ is dense in $\mathcal{H}_{\Theta b}(E)$. By Lemma
\ref{L identidade ss B não constante}, $L(e^{\phi})=\mathcal{B}(\Gamma_{\Theta}(L))(\phi)\cdot e^{\phi}$ for every $\phi \in E'$,  and this implies that each $e^{\phi}$ belongs to the range of $L$.
Therefore,
\[
\mathcal{H}_{\Theta b}(E)=\overline{{\rm span}\{e^{\phi}: \phi\in E^{\prime}%
\}}^{\,\tau_{\theta}}=\overline{L(\mathcal{H}_{\Theta b}(E)))}^{\,\tau_{\theta}}.%
\]

\end{proof}

We can go farther with $\pi_1$-$\pi_2$-holomorphy types. The following result is closely related to a result of Malgrange \cite{malgrange} on the existence of solutions of convolution equations. Its proof follows the sames steps of the proof of \cite[Theorem 4.4]{favaro-jatoba1}:

\begin{theorem}
\label{conv_op_surjective}Let $(\mathcal{P}_{\Theta}(^{n}E))_{n=0}^{\infty
}$ be a $\pi_{1}$-$\pi_2$-holomorphy type from $E$ to $\mathbb{C}$ such that ${\rm Exp}_{\Theta^{\prime}}(E^{\prime})$ is closed under division, that is: if $f,g \in {\rm Exp}_{\Theta^{\prime}}(E^{\prime})$ are such that $g \neq 0$ and $f/g$ is holomorphic, then $f/g \in {\rm Exp}_{\Theta^{\prime}}(E^{\prime})$. Then every nonzero convolution
operator on $\mathcal{H}_{\Theta b}(E)$ is surjective.
\end{theorem}

\begin{example}
\rm (a) Let $E^{\prime}$ have the bounded
approximation property and $\left(\mathcal{P}_{N}(^{m}E)\right)_{m=0}^{\infty}$ be the holomorphy type of nuclear homogeneous polynomials on $E$. To see that this is a $\pi_{1}$-$\pi_{2}$-holomorphy type, regard it as a particular case of the quasi-nuclear holomorphy types considered in Example \ref{exemplos}(a) or see it directly in \cite[page 15 and Lemma 7.2]{apostilagupta}. By \cite[Proposition 7.2]{apostilagupta}, in this case the role of ${\rm Exp}_{\Theta^{\prime}}(E^{\prime})$ is played by the space ${\rm Exp}(E^{\prime})$ of all  entire mappings of exponential-type on $E'$ \cite[Definition 7.5]{apostilagupta}. 
Also, ${\rm Exp}(E^{\prime})$ is closed under division \cite[Proposition 8.1]{apostilagupta}. Hence, it follows from Theorem \ref{conv_op_surjective} that every nonzero convolution operator on $\mathcal{H}_{N b}(E)$ is surjective.

\medskip

\noindent(b) As we saw in Example \ref{exemplos}(a), if $E^{\prime}$ has the bounded
approximation property, then $\left(\mathcal{P}_{\widetilde{N},\left( s;\left( r,q\right)
\right)}(^{m}E)\right)_{m=0}^{\infty}$ is a $\pi_{1}$-$\pi_{2}$-holomorphy type, and, according to the duality (\ref{duality}), in this case the role of ${\rm Exp}_{\Theta^{\prime}}(E^{\prime})$ is played by ${\rm Exp}_{\left( s^{\prime},m\left(
r^{\prime};q^{\prime}\right) \right)}(E^{\prime})$. 
Making $A=B=0$ in \cite[Theorem 3.8]{belgica} one gets that ${\rm Exp}_{\left( s^{\prime},m\left(
r^{\prime};q^{\prime}\right) \right)}(E^{\prime})$ is closed
under division (alternatively, see
 \cite[Theorem 5.4.8]{matos7}). Hence, it follows from Theorem \ref{conv_op_surjective} that every nonzero convolution operator on $\mathcal{H}_{\widetilde{N},\left( s;\left( r,q\right)
\right) b}(E)$ is surjective.
\end{example}

Now we establish a connection with the fashionable subject of lineability (for detailed information see, e.g., \cite{seoane} and references therein).

\begin{definition}\rm A subset $A$ of an infinite-dimensional topological vector space $E$ is said to be {\it dense-lineable in $E$} if $A \cup \{0\}$ contains a dense subspace of $E$.
\end{definition}

Next result is closely related to (actually is a generalization of) \cite[Corollary 12]{aronmarkose}:

\begin{proposition}
\label{dense-lineable}Let $E'$ be separable, $(\mathcal{P}_{\Theta}(^{n}E))_{n=0}^{\infty
}$ be a $\pi_{1}$-holomorphy type from $E$ to $\mathbb{C}$ and $L$ be a convolution operator on $\mathcal{H}_{\Theta b}(E)$ that is not a scalar multiple of the identity. Then the set of hypercyclic functions for $L$ is dense-lineable in $\mathcal{H}_{\Theta b}(E)$.
\end{proposition}


\begin{proof}
The convolution operator $L$ is hypercyclic by Theorem \ref{main1}, so we can take a hypercyclic function $f$
for $L$. Define%
\[
M=\left\{  \sum\limits_{i=0}^{m}\lambda_{i}L^{i}\left(  f\right)  :m\in\mathbb{N}_{0}, \lambda_0, \lambda
_{1}, \ldots, \lambda_m \in\mathbb{C}\right\}  ,
\]
where $L^{0}$ denotes the identity on $\mathcal{H}_{\Theta b}(E).$ Clearly $M$ is a vector subspace of $\mathcal{H}_{\Theta b}(E)$ and,
since $\left\{  L^{n}\left(  f\right)  :n\in\mathbb{N}_{0}\right\}  $ is
contained in $M$, $M$ is a dense subset of $\mathcal{H}_{\Theta b}(E)$. Now we only have to prove that every nonzero element of $M$ is hypercyclic for $L$, that is,  for every $g\in M,$ $g\neq0,$
the set $\left\{  g,L\left(  g\right)  ,\ldots,L^{n}\left(
g\right)  ,\ldots\right\}  $ is dense in $\mathcal{H}_{\Theta b}(E).$ If $g\in M,$ $g\neq0,$ then $g=\sum\limits_{i=0}^{m}\lambda_{i}L^{i}(f)$, with $\lambda_{0},\lambda_{1},\ldots,\lambda_{m}\in\mathbb{C}.$ Note that $\sum\limits_{i=0}^{m}\lambda_{i}L^{i}\neq0$ because $g\neq0$. Since $\sum\limits_{i=0}%
^{m}\lambda_{i}L^{i}$ is a convolution operator, it follows from Proposition \ref{dense_range} that $\sum\limits_{i=0}^{m}\lambda_{i}L^{i}$ has dense range.
Using that $\left\{  L^{n}\left(  f\right)  :n\in\mathbb{N}_{0}\right\}  $ is dense
in $\mathcal{H}_{\Theta b}(E)$ and that $\sum\limits_{i=0}^{m}\lambda_{i}L^{i}$ is
continuos and has dense range, we conclude that the set%
\[
\sum\limits_{i=0}^{m}\lambda_{i}L^{i}\left(  \left\{  L^{n}\left(  f\right)
:n\in\mathbb{N}_{0}\right\}  \right)
\]
is dense in $\mathcal{H}_{\Theta b}(E).$ So,
\begin{align*}
\{g, L(g), \ldots, L^n(g), \ldots, \} &=\left\{  L^{n}\left(g\right)  :n\in\mathbb{N}_{0}\right\}\\&=\left\{  L^{n}\left(  \sum\limits_{i=0}^{m}\lambda_{i}L^{i}\left(  f\right)
\right)  :n\in\mathbb{N}_{0}\right\}\\&  =\sum\limits_{i=0}^{m}\lambda_{i}%
L^{i}\left(  \left\{  L^{n}\left(  f\right)  :n\in\mathbb{N}_{0}\right\}
\right)
\end{align*}
is dense in $\mathcal{H}_{\Theta b}(E)$, proving that $g$ is hypercyclic for $L$.
\end{proof}

Combining Theorem \ref{main2} with Proposition \ref{dense-lineable} we get:

\begin{corollary}
Let $E^{\prime}$ be
separable, $(\mathcal{P}_{\Theta}(^{m}E))_{m=0}^{\infty}$ be a
$\pi_1$-$\pi_2$-holomorphy type and $T\in[\mathcal{H}_{\Theta
b}(E)]^{\prime}$ be a linear functional which is not a scalar multiple of $\delta_0$. Then the set of hypercyclic functions for $\bar{\Gamma}_{\Theta}(T)$ is dense-lineable in $\mathcal{H}_{\Theta b}(E)$.
\end{corollary}

A result similar to \cite[Proposition 4.1]{CDSjmaa} is the following:

\begin{proposition}
\label{unique_functional}Let $(\mathcal{P}_{\Theta}(^{m}E))_{m=0}^{\infty
}$ be a $\pi_{1}$-holomorphy type from $E$ to $\mathbb{C}$. Then for every convolution operator $L\colon\mathcal{H}_{\Theta
b}(E)\longrightarrow\mathcal{H}_{\Theta b}(E),$ the functional $\Gamma_\Theta (L)$ is the unique functional in $[\mathcal{H}_{\Theta
b}(E)]^{\prime}$ such that $L(f)=\Gamma_\Theta (L)\ast f$ for every $f\in \mathcal{H}_{\Theta b}(E)$.
\end{proposition}

\begin{proof}
Let $L\colon\mathcal{H}_{\Theta b}(E)\longrightarrow\mathcal{H}_{\Theta b}(E)$ be a convolution operator. By the definition of $\Gamma_\Theta$ we have that $\Gamma_\Theta\in [\mathcal{H}_{\Theta b}(E)]^\prime$ and%
\begin{align*}
L(f)(x)&=L(f)(0-(-x))=[\tau_{-x} L(f)](0)\\&=L(\tau_{-x} f)(0)=\Gamma_\Theta(L)(\tau_{-x}f)\\&=\Gamma_\Theta(L)\ast f(x)
\end{align*}
for all $f\in \mathcal{H}_{\Theta b}(E)$ and $x\in E$. Thus, $L(f)=\Gamma_\Theta (L)\ast f$ for every $f\in \mathcal{H}_{\Theta b}(E)$. Let us prove the uniqueness: if $S\in[\mathcal{H}_{\Theta
b}(E)]^{\prime}$ is such that $L(f)=S\ast f$, then
\[
L(e^{\phi})(x)=S\ast e^{\phi}(x)=S(\tau_{-x}e^{\phi})=S(e^{\phi})\cdot
e^{\phi(x)}=\mathcal{B}S(\phi)\cdot e^{\phi(x)}
\]
for all $\phi\in E^\prime$ and $x\in E$. Hence $L(e^\phi)=\mathcal{B}S(\phi)\cdot e^{\phi} $ for every $\phi\in E^\prime$. It follows from Lemma \ref{L identidade ss B não constante}(a) that $\mathcal{B}(\Gamma_{\Theta}(L))(\phi)=\mathcal{B}S(\phi)$ for every $\phi\in E^\prime$. So $S=\Gamma_{\Theta}(L)$ by the injectivity of the Borel transform (Proposition \ref{Propo isomorfismo Hthetab}).
\end{proof}



We finish the paper exploring the multiplicative structure of $[\mathcal{H}_{\Theta b}(E),\tau_{\Theta}%
]^{\prime}$:

\begin{definition}\rm
\textrm{Let $(\mathcal{P}_{\Theta}(^{m}E))_{m=0}^{\infty}$ be a $\pi_{2}%
$-holomorphy type from $E$ to $\mathbb{C}$. For $T_{1},T_{2}\in\lbrack\mathcal{H}_{\Theta
b}(E)]^{\prime}$ we define the \emph{convolution product} of $T_{1}$  and
$T_{2}$ in $[\mathcal{H}_{\Theta b}(E)]^{\prime}$ by%
\[
T_{1}\ast T_{2}:=\Gamma_{\Theta}({O}_{1}\circ {O}_{2}%
)\in\lbrack\mathcal{H}_{\Theta b}(E)]^{\prime},
\]
where $O_{1}=T_{1}\ast$ and $O_{2}=T_{2}\ast$. }
\end{definition}
}
It is easy to see that $[\mathcal{H}_{\Theta b}(E)]^{\prime}$ is an algebra
under this convolution product with unity $\delta_0$. Furthermore, the convolution product satisfies the following property:
$$\left(  T_{1}\ast T_{2}\right)  \ast
f=T_{1}\ast\left(  T_{2}\ast f\right),  $$
for all $T_{1},T_{2}\in\lbrack\mathcal{H}_{\Theta
b}(E)]^{\prime}$ and $f \in \mathcal{H}_{\Theta b}(E)$.

The same proof of \cite[Theorem 3.3]{favaro-jatoba1} provides the following analogue of \cite[Corollary 4.2]{CDSjmaa}:

\begin{theorem}
\label{isomorfismo algebras}If $(\mathcal{P}_{\Theta}(^{m}E))_{m=0}^{\infty}$
is a $\pi_{1}$-$\pi_{2} $-holomorphy type, then the Borel transform is an
algebra isomorphism between $[\mathcal{H}_{\Theta b}(E),\tau_{\Theta}%
]^{\prime}$ and ${\rm Exp}_{\Theta^{\prime}}(E^{\prime}).$
\end{theorem}

\bigskip

\noindent Faculdade de Matem\'atica\\
Universidade Federal de Uberl\^andia\\
38.400-902 - Uberl\^andia - Brazil\\
e-mails: bertoloto@famat.ufu.br, botelho@ufu.br,
vvfavaro@gmail.com,\\ marques@famat.ufu.br

\end{document}